\documentclass[runningheads]{llncs}
\usepackage{url}
\usepackage{multirow}
\usepackage{amssymb}
\usepackage{amsxtra}
\usepackage{makeidx}
\usepackage{amsfonts}
\usepackage{mathtools}
\usepackage{algorithm}
\usepackage{algorithmic}
\usepackage{float}
\usepackage{graphicx}
\usepackage{comment}
\usepackage[font=small,labelfont=bf]{caption}
\usepackage{subcaption}
\usepackage{breqn}
\usepackage[russian]{babel}

\usepackage{color}
\usepackage{hyperref}

\DeclareUnicodeCharacter{202A}{\,}

\begin{document}

\title{Численные методы для некоторых классов относительно сильно монотонных вариационных неравенств}
\titlerunning{Методы для относительно сильно монотонных вариационных неравенств}

\author{ Ф.\,C. Стонякин$^{1,2}$, А.\,А. Титов$^{1,3}$, Д.\,В. Макаренко$^1$, М.\,C. Алкуса$^{1,3}$ 
\thanks{Исследования Ф.С. Стонякина в пyнкте 2, а также исследования А.А. Титова по доказательству леммы 2 и теоремы 2 выполнены при поддержке гранта Российского наyчного фонда, проект 21-71-30005.}}

\authorrunning{Ф.\,C. Стонякин, А.\,А. Титов, Д.\,В. Макаренко, М.\,C. Алкуса}	

\institute{Московский физико-технический институт, Долгопрудный
	\and
	Крымский федеральный университет им. В.И. Вернадского, Симферополь\\
	\and 
	Национальный исследовательский университет «Высшая школа экономики», Москва\\
	\email{fedyor@mail.ru, a.a.titov@phystech.edu, devjiu@gmail.com,
		mohammad.alkousa@phystech.edu}
	}
	
\maketitle
\begin{abstract}
Статья посвящена существенному расширению недавно предложенного класса относительно сильно выпуклых оптимизационных задач в пространствах больших размерностей. В работе вводится аналог понятия относительной сильной выпуклости для вариационных неравенств (относительная сильная монотонность) и исследуются оценки скорости сходимости некоторых  численных методов первого порядка для задач такого типа. В статье рассматриваются два класса вариационных неравенств в зависимости от условий, связанных с гладкостью оператора. Первый из этих классов задач включает в себя относительно ограниченные операторы, а второй --- операторы с аналогом условия Липшица (так называемая относительная гладкость). Для вариационных неравенств с относительно ограниченными и относительно сильно монотонными операторами была исследована вариация субградиентного метода и обоснована оптимальная оценка скорости сходимости. Для задач с относительно гладкими и относительно сильно монотонными операторами доказана линейная скорость сходимости алгоритма со специальной организацией процедуры рестартов (перезапусков) проксимального зеркального метода для вариационных неравенств с монотонными операторами. 

\end{abstract}

\section{Введение и постановка задачи}\label{sec_introduction}

Численные методы градиентного типа достаточно часто используются для самых разнообразных постановок задач выпуклой оптимизации в пространствах больших размерностей. Это объясняется небольшими затратами памяти на итерациях, а также возможностью обоснования приемлемых оценок скорости сходимости, не содержащих (в отличие, например, от методов отсекающей гиперплоскости) параметров размерности пространства. Однако при этом существенны предположения о функциональных свойствах таких задач (гладкость, липшицевость, сильная выпуклость). Так, несколько лет назад был выделен класс относительно гладких задач оптимизации (см., например \cite{Bauschke,Drag,Lu_Nesterov_2018}). Свойство относительной $L$-гладкости ($L > 0$) обобщает ycловие $L$-гладкости ($L$-липшицевости градиента)  $f$ путём замены в справедливом для $L$-гладкий фyнкций $f$ неравенстве ($Q$ --- область определения $f$)
$$
	f(y) \leq f(x) + \langle \nabla{f(x)}, y - x \rangle  + \frac{L}{2} \|x - y \|_2^2 \quad   \forall x, y \in Q
$$	
выражения $\frac{1}{2} \|x - y \|_2^2 $ дивергенцией (расхождением) Брэгмана (см. \eqref{Brg_form} и \eqref{funct_rel_smooth} ниже), которая порождается некоторой выпуклой прокс-функцией (важно, что она не обязательно сильно выпукла). Отметим, что здесь и всюдy далее $\|\cdot\|_2$ --- евклидова норма в $n$-мерном пространстве $\mathbb{R}^n$.

Для выпyклых относительно гладких задач которых оценки сходимости обычных (неускоренных) методов градиентного типа оптимальны с точностью до умножения на константу, не зависящую от размерности и параметров метода (см. работы \cite{Bauschke,Drag,Dragomir,Lu_Nesterov_2018}, а также имеющиеся в них ссылки). В работе \cite{Lu_Nesterov_2018} введено понятие относительной сильной выпуклости функции, которое позволило расширить класс выпуклых оптимизационных задач, для которых можно доказать линейную скорость сходимости (сходимость со скоростью геометрической прогрессии) метода градиентного типа, причём соответствующая оценка не содержит параметров размерности задачи. В данной работе мы развиваем этот подход и исследуем некоторые алгоритмы уже для вариационных неравенств с аналогом относительной сильной выпуклости для операторов (относительной сильной монотонностью). Напомним, что понятие относительной сильной выпуклости \cite{Lu_Nesterov_2018} функции $f$ обобщает понятие обычной $\mu$-сильной выпуклости $f$ ($\mu > 0$) путём замены в неравенстве 
\begin{equation}
	f(x) + \langle \nabla{f(x)}, y - x \rangle  + \frac{\mu}{2} \|x - y \|_2^2 \leq f(y) \quad   \forall x, y \in Q,
	\end{equation}
выражения $\frac{1}{2} \|x - y \|_2^2 $ дивергенцией Брэгмана (см. \eqref{Brg_form} и \eqref{eqrelativestorngconv} ниже), которая порождается некоторой выпуклой прокс-функцией. 

В настоящей статье рассматриваются методы первого порядка для двух классов вариационных неравенств с операторами, удовлетворяющими предлагаемому аналогу условия относительной сильной выпуклости (см. ниже определение  \ref{DefRelStrongMonot} относительной сильной монотонности оператора): с аналогом ограниченности (относительная ограниченность, см. определение 2 ниже), а также с аналогом условия Липшица (относительная гладкость, см. определение 3 ниже).

Хорошо известно, что на классе липшицевых и сильно выпуклых минимизационных задач оптимальная оценка скорости сходимости достигается именно для субградиентного метода \cite{Simon_Julien_Bach_2012}. В последние годы активно исследуются задачи с аналогом условия Липшица относительно некоторой выпуклой прокс-функции (относительная липшицевость), которая, в отличие от классической постановки, не обязана удовлетворять условию сильной выпуклости относительно нормы \cite{AdaMirr_2021,Lu_2018,Zhou_NIPS_2020}. Мы исследуем оценку скорости сходимости субградиентного метода для сильно выпуклых задач с аналогичным предположением об относительной липшицевости. Точнее говоря, в настоящей статье рассматривается вариант субградиентного метода на классе относительно ограниченных и относительно сильно монотонных вариационных неравенств, а также класс относительно сильно выпукло-вогнутых седловых задач с соответствующими условиями относительной липшицевости функционалов. 

Далее, немалую популярность в работах по оптимизации получило упомянутое выше недавно предложенное понятие относительной гладкости функций (см. работы \cite{Bauschke,Drag,Dragomir,Lu_Nesterov_2018}, а также приведённые в них ссылки), которое позволило существенно расширить класс задач выпуклой оптимизации по сравнению со стандартным предположением о липшицевости градиента с гарантией оценки скорости сходимости $O(N^{-1})$ (здесь и далее $N$ --- количество итераций), которая может считаться оптимальной для такого широкого класса задач \cite{Dragomir}. В плане приложений можно отметить подход к построению методов градиентного типа для задач распределенной оптимизации с использованием относительной гладкости и относительной сильной выпуклости \cite{Hendr}. Аналоги относительной гладкости введены в последние пару лет и для более общей постановки задачи решения вариационного неравенства (см. \cite{Inex}, а также имеющиеся там ссылки) с монотонным оператором. Оказывается, что для этого класса задач можно предложить алгоритмы экстраградиентного типа с гарантией оценки скорости сходимости $O(N^{-1})$. Мы же рассматриваем класс относительно сильно монотонных и относительно гладких операторов и показываем, как некоторая вариация адаптивного проксимального зеркального метода \cite{UMP} со специальной организацией процедуры рестартов (перезапусков) может приводить к обоснованию лучшей оценки скорости (линейной скорости сходимости) для вариационных неравенств с такими предположениями. Стоит отметить, что метод адаптивен, т.е. в оценке скорости сходимости глобальный параметр относительной гладкости задачи можно заменить его адаптивно подбираемыми на итерациях потенциально более удобными локальными аппроксимациями. Предлагаемый нами подход, в частности, дал возможность впервые предложить метод с адаптивной настройкой на итерациях параметра относительной гладкости для выделенного в \cite{Hendr} класса задач распределённой оптимизации.

Работа состоит из введения, трех основных частей (пунктов) и заключения. Второй пункт статьи посвящён модификации метода зеркального спуска и выводу оценки его скорости сходимости для вариационных неравенств с относительно сильно монотонными и относительно ограниченными операторами. В частности, полученная оценка указывает на оптимальность такого метода на выделенном классе вариационных неравенств, поскольку она оптимальна (с точностью до умножения на не зависящую от параметров метода и размерности пространства константу) даже на более узком классе задач минимизации относительно липшицевых и относительно сильно выпуклых функций \cite{Lu_2018}. В третьем пункте статьи рассматривается класс относительно сильно монотонных и относительной гладких операторов и анализируется возможность использования рестартованного адаптивного проксимального зеркального метода для такого класса задач с обоснованием гарантии линейной скорости сходимости. В четвертом пункте показывается, как предложенные ранее алгоритмы для вариационных неравенств и полученные теоретические оценки их скорости сходимости могут быть применены для решения относительно сильно выпукло-вогнутых седловых задач с соответствующими предположениями о гладкости функционалов.

Будем рассматривать задачу нахождения решения $x_*$ (также называемого слабым решением) вариационного неравенства: 
\begin{equation}\label{eq:1}
\max_{x \in Q} \langle g(x), x_* - x \rangle \leq 0,
\end{equation}
где $Q$ --- выпуклое замкнутое подмножество $\mathbb{R}^n$,
$g: Q \longrightarrow \mathbb{R}^n$. Предположим, что удовлетворяющее \eqref{eq:1} решение $x_*$ существует.

Всюду далее будем предполагать, что нам доступна некоторая выпуклая (вообще говоря, не сильно выпуклая) дифференцируемая прокс-функция $d$, порождающая расстояние, а также соответствующая ей дивергенция (расхождение) Брэгмана \cite{Bauschke}
\begin{equation}\label{Brg_form}
V(y, x) = d(y) - d(x) - \langle \nabla d(x), y - x \rangle.
\end{equation}

Введём следующий аналог понятия относительной сильной выпуклости функции \cite{Lu_Nesterov_2018} для вариационных неравенств.
\begin{definition}\label{DefRelStrongMonot}
Назовём оператор $g$ относительно $\mu$-сильно монотонным, где $\mu >0$, если для всяких $x, y \in Q$ верно неравенство
	\begin{equation}\label{eq:3}
	 	\mu V(y, x) + \mu V(x, y) \leq \langle g(y) - g(x), y - x \rangle.
	 \end{equation}
\end{definition}
Как правило, далее в статье мы будем использовать следующее неравенство, естественно вытекающее из \eqref{eq:3}.
\begin{remark}
Если оператор $g$ является  относительно $\mu$-сильно монотонным, то для всяких $x, y \in Q$ верно неравенство
$$
	 	\mu V(x, y) \leq \langle g(y) - g(x), y - x \rangle.
$$
\end{remark}

Поясним на примере, почему относительная сильная монотонность вводится име\-нно согласно  \eqref{eq:3}.
\begin{example}
	Если $f: Q \longrightarrow \mathbb{R}$ --- относительно $\mu$-сильно выпуклая функция
	\begin{equation}\label{eqrelativestorngconv}
	f(x) - f(y) + \mu V(x, y) \leq \langle \nabla{f(x)}, x - y \rangle \quad   \forall x, y \in Q,
	\end{equation}
	то
	\begin{equation}
	f(y) - f(x) + \mu V(y, x) \leq \langle \nabla{f(y)}, y - x \rangle \quad   \forall x, y \in Q.
	\end{equation}
	После сложения двух последних неравенств получаем
	\begin{align*}
	\mu V(x, y) + \mu V(y, x)\leq \langle \nabla{f(y)} - \nabla{f(x)}, y - x \rangle \quad  \forall x, y \in Q.
	\end{align*}
	Таким образом, неравенство \eqref{eq:3} верно при $g(x) = \nabla{f(x)}$, где $\nabla{f(x)}$ --- произвольный субградиент $f$.
\end{example}

Относительно сильно выпуклые функционалы возникают в самых разных ситуациях, среди которых мы упомянем задачу централизованной распределённой минимизации эмпирического риска в предположении схожести слагаемых \cite{Hendr}.

 \begin{example}{Минимизация эмпирического риска \cite{Hendr}.}\label{min_risk}\\
	Рассмотрим задачу минимизации эмпирического риска
	\begin{equation}\label{EmpirProbl}
F(x)=\frac{1}{m} \sum_{j=1}^{m} f_{j}(x)=\frac{1}{n m} \sum_{j=1}^{m} \sum_{i=1}^{n} \ell\left(x, z_{i}^{(j)}\right)
\rightarrow\min\limits_{x\in Q},
	\end{equation}
	\begin{equation}
	F(x)=\frac{1}{N}\sum\limits_{i=1}^N \ell(x,z_i),
	\end{equation}
	в предположении, что исходные данные представляют из себя набор $n$ выборок, каждая из которых хранится на одном из $m$ серверов. При этом для достаточно большого $n$ все $f_j$ есть $\mu$-сильно выпуклые и $L$-гладкие (удовлетворяют условию Липшица градиента с константой $L > 0$) функционалы, которые можно считать статистически схожими \cite{Hendr}. Такая схожесть может быть описана в виде предположения \cite{Hendr} о том, что для всякого $x$
	$$
	\|\nabla^2 F(x) - \nabla^2 f_j (x)\|_2 \leq \delta
	$$
	при всяком $j$ для некоторого достаточно малого $\delta >0$, причём здесь и всюдy далее $\|A\|_2 = \max\limits_{\|x\|_2 \leq 1}\|Ax\|_2$. При этом предполагается, что существует центральный сервер (ему соответствует функционал $\overline{f}$), на который передаётся информация о градиентах $f_j$ в текущей точке, но не передаётся информация о значениях $f_j$. В \cite{Hendr} показано, что при таком допущении можно ввести прокс-функцию
	\begin{equation}\label{prox_risk}
	d(x):= \overline{f}(x) + \frac{\delta}{2}\|x\|_2^2 
	\end{equation}
	и для соответствующей дивергенции Брэгмана $V(y, x) = d(y) - d(x) - \langle \nabla d(x), y -x \rangle$ фyнкция $F$ будет относительно $1$-гладкой и относительно $\frac{\mu}{\mu + 2\delta}$-сильно выпуклой, т.е. для всяких $x, y \in Q$ верны неравенства
$$
F(y) \leq F(x) + \langle \nabla F(x), y - x\rangle + V(y, x)
$$
и
$$
F(y) \geq F(x) + \langle \nabla F(x), y - x\rangle + \frac{\mu}{\mu + 2\delta} V(y, x).
$$
	
Это означает, что при $\delta \ll L$ можно улучшить оценку скорости сходимости неускоренного градиентного метода $$O\left(\frac{L}{\mu}\log\frac{1}{\varepsilon}\right)$$
для задач минимизации эмпирического риска (целевой функционал $L$-гладкий и $\mu$-сильно выпуклый) методами первого порядка до $$O\left(\left(1+\frac{\delta}{\mu}\right)\log\frac{1}{\varepsilon}\right).$$
	
Отметим, что при сопоставимых значениях параметров $\delta$ и $\mu$ (этого ввиду малости $\delta$ возможно добиться, к примеру, регуляризацией задачи) такая оценка близка к оценке сложности ускоренных методов
	$$O\left(\sqrt{1+\frac{\delta}{\mu}}\log\frac{1}{\varepsilon}\right).$$
При этом на каждой итерации метода, выполняемого центральным узлом, доступна информация о градиенте целевого функционала $F$, но не доступна информация о значении функционала $F$. Доступность информации о градиенте позволяет рассматривать поставленную задачу минимизации эмпирического риска как задачу отыскания решения вариационного неравенства с относительно гладким и относительно сильно монотонным оператором $g = \nabla F$. Такой подход позволяет, в частности, предложить метод с полной адаптивной настройкой (см. п. \ref{sect_UMP} далее) на параметр гладкости задачи \eqref{EmpirProbl}, чего не удалось добиться для метода в \cite{Hendr}. 
\end{example}

В рамках данной статьи мы рассмотрим численные методы решения вариационных неравенств с операторами, удовлетворяющими условиям относительной ограниченности, а также относительной гладкости.
\begin{definition}\label{DefRelBound}\cite{Main}
Назовём оператор $g: Q \longrightarrow \mathbb{R}^n$ относительно $M$-огранич\-енным, где $M >0$, если для всяких $x, y \in Q$ верно неравенство
	\begin{equation}\label{rel_bound}
	 	\langle g(x), x - y \rangle \leq M\sqrt{2V(y,x)}.
	 \end{equation}
\end{definition}
\begin{definition}\cite{Inex}
Назовём оператор $g: Q \longrightarrow \mathbb{R}^n$ относительно $L$-гладким, где $L > 0$, если для всяких $x, y, z \in Q$ верно неравенство
\begin{equation}\label{rel_smooth}
    \langle g(y)-g(z),x-z\rangle \leq LV(x,z) + LV(z,y).
    \end{equation}
\end{definition}
Отметим, что если функция $f$ $L$-относительно гладкая \cite{Bauschke}, т.е.
\begin{equation}\label{funct_rel_smooth}
f(y) \leq f(x) + \langle \nabla f(x), y - x\rangle + LV(y, x) \quad \forall x, y \in Q,
\end{equation}
то оператор $g(x) = \nabla f(x)$ yдовлетворяет \eqref{rel_smooth}. Однако в слyчае непотенциального оператора $g$ yсловие \eqref{rel_smooth} не сводится, вообще говоря, к \eqref{funct_rel_smooth} для какой-нибyдь фyнкции $f$.

\section{Субградиентный метод для вариационных неравенств с относительно сильно монотонными и ограниченными операторами}\label{section_2}

Вслед за \cite{Simon_Julien_Bach_2012} предложим метод зеркального спуска \eqref{eq:4}, но уже для рассматриваемого в настоящей работе класса  вариационных неравенств с относительно сильно монотонными и относительно ограниченными операторами (определения \ref{DefRelStrongMonot} и \ref{DefRelBound}):

\begin{equation} \label{eq:4}
x_{k+1} := \arg \min_{x \in Q} \left\{ h_k \langle g(x_k), x \rangle + V(x, x_k)\right\},
\end{equation}
где
$$
    h_k = \frac{2}{\mu(k+1)}\quad  \forall k= 0,1, 2, \ldots.
$$

Непосредственно можно проверить следующий вспомогательный результат для шага метода зеркального спуска \eqref{eq:4}.

\begin{lemma}\label{th:base}
Если для $g$ верно \eqref{rel_bound}, а $x_k$ и $x_{k+1}$ удовлетворяют \eqref{eq:4}, то для произвольного $x \in Q$ верно неравенство
$$	
h_k \langle g(x_k), x_k - x \rangle \leq \frac{h_k^2 M^2}{2} + V(x, x_k) - V(x, x_{k+1}).
$$
\end{lemma}
\begin{proof}
Если применить стандартное условие экстремума 1-го порядка ко вспомогательной минимизационной подзадаче \eqref{eq:4}, то можно проверить при $h_k >0$ для всякого $x\in Q$ справедливость неравенств
$$h_k \langle g(x_k), x_k - x \rangle \leq h_k \langle g(x_k), x_k - x_{k+1} \rangle  + V(x, x_k) - V(x, x_{k+1}) -V(x_{k+1},x_k) \stackrel{\eqref{rel_bound}}{\leq}$$
$$
\leq h_kM\sqrt{2V(x_{k+1},x_k)}+ V(x, x_k) - V(x, x_{k+1}) -V(x_{k+1},x_k) \leq
$$
$$
\leq \frac{h_k^2M^2}{2} + V(x, x_k) - V(x, x_{k+1}).
$$
\end{proof}

Согласно лемме \ref{th:base}, получим, что при всяких $ k \geq 0$ и $x \in Q$ верно
\begin{equation} 
\langle g(x_k), x_k - x \rangle \leq \frac{h_k M^2}{2} + \frac{V(x, x_k)}{h_k} - \frac{V(x, x_{k+1})}{h_k}. 
\end{equation}

Далее, с учетом \eqref{eq:3}, получим 
\begin{equation*}
\langle g(x_k), x_k - x \rangle \geq  \langle g(x), x_k - x \rangle + \mu (V(x, x_k) + V(x_k, x)) \quad \forall x \in Q,
\end{equation*}
откуда при всяком $k \ge 0$ имеем:
\begin{equation}
\begin{aligned} 
2k\langle g(x), x_k - x \rangle +  2k\mu (V(x, x_k) + V(x_k, x)) &\leq  
\frac{2k M^2}{\mu (k+1)} + \mu k (k+1)V(x, x_k) -  \\&
 - \mu k (k+1)V(x, x_{k+1}) \quad \forall x \in Q. 
\end{aligned}
\end{equation}
Это означает, что
\begin{equation}\label{eq:5}
\begin{aligned} 
2k\langle g(x), x_k - x \rangle +  2k\mu V(x_k, x) \leq   
\frac{2k M^2}{\mu (k+1)} &+ \mu k (k-1)V(x, x_k) -  \\& -
 \mu k (k+1)V(x, x_{k+1}) \quad  \forall x \in Q. 
\end{aligned}
\end{equation}
Пусть алгоритм \eqref{eq:4} отработал $N$ шагов. Тогда можно просуммировать неравенства \eqref{eq:5} по $k$ от $1$ до $N$ и учесть, что $\frac{k}{k+1} \le 1$:
\begin{equation}
\sum_{k=1}^{N} 2k(\langle g(x), x_k - x \rangle) + \mu V(x_k, x) \leq \frac{2NM^2}{\mu},
\end{equation}
откуда с учетом $2(1+2+...+N)=N(N+1)$:
\begin{equation} \label{eq:122}
\sum_{k=1}^{N} \frac{2k}{N(N+1)}(\langle g(x), x_k - x \rangle) + \mu V(x_k, x)) 
\leq \frac{2M^2}{\mu(N+1)} \quad \forall x \in Q.
\end{equation}

Если учесть, что $V(x_k, x) \geq 0 \quad \forall x \in Q, \; \forall k \ge 0$,
то при
$$\widehat{x} = \sum_{k=1}^{N} \frac{2 k}{N (N+1)} x_k$$
будет верно неравенство 
\begin{equation} \label{eq:13}
\max_{x \in Q} \langle g(x), \widehat{x} - x \rangle \leq \frac{2 M^2}{\mu (N+1)} \leq \varepsilon,
\end{equation}
после $N = O\left(\frac{M^2}{\mu \varepsilon}\right)$ итераций алгоритма $\eqref{eq:4}$. Как известно, такая оценка сложности оптимальна даже на классе относительно липшицевых и относительно сильно выпуклых минимизационных задач \cite{Lu_2018}. Это указывает на её оптимальность и для существенно более широкого класса относительно липшицевых и относительно сильно выпуклых минимизационных задач, а значит и для рассматриваемого класса вариационных неравенств. 

Таким образом, можно сформулировать следующий результат.
\begin{theorem}\label{thm_MD_VI}
  Пусть $g$ --- относительно $\mu$-сильно монотонный и $M$-относитель\-но ограниченный оператор. Тогда после $N$ итераций алгоритма: 
	$$ x_{k+1} := \arg \min_{x \in Q} \{ h_k \langle g(x_k), x\rangle + V(x, x_k)\}, \;\;\; h_k = \frac{2}{\mu (k+1)}$$
	будет верно неравенство:
	\begin{equation}\label{eq:2}
	\max_{x \in Q} \langle g(x), \widehat{x} - x\rangle \leq \frac{2 M^2}{\mu (N+1)},
	\end{equation}
		где 
	$$
	\widehat{x} = \sum_{k=1}^{N} \frac{2 k}{N (N+1)} x_k.
	$$
\end{theorem}

\begin{remark}
Если $x_*$ --- сильное решение рассматриваемого вариационного не\-равенства, то можно выписать оценку скорости сходимости и <<по аргументу>>, поскольку $\langle g(x_*), x_k - x_*\rangle \geq 0$. Тогда \eqref{eq:122} означает, что 
	\begin{equation} \label{eq:12}
	\begin{aligned} 
	\sum_{k=1}^{N} \frac{2k\mu V(x_k, x_*)}{N(N+1)} \leq \frac{2M^2}{\mu(N+1)} \quad  \forall x \in Q,
	\end{aligned}
	\end{equation}
Если же прокс-функция $1$-сильно выпукла относительно нормы $\|\cdot\|$, то из \eqref{eq:12} вытекает следующая оценка:
	\begin{equation} 
	\begin{aligned} 
	\|x_* - x_k\|^2 \leq \frac{4M^2}{\mu(N+1)}.
	\end{aligned}
	\end{equation}
\end{remark}
\begin{remark}
Если прокс-функция $d$ является $1$-сильно выпуклой, то неравенство из леммы 1 можно уточнить:
\begin{equation} 
h_k \langle g(x_k), x_k - x \rangle \leq \frac{h_k^2 \|g(x_k)\|_*^2}{2} + V(x, x_k) - V(x, x_{k+1}). 
\end{equation}
Тогда итоговая оценка \eqref{eq:13} приобретает следующий вид:
\begin{equation}
\max_{x \in Q} \langle g(x), \widehat{x} - x \rangle \leq \frac{2}{\mu N (N+1)} \sum_{k=1}^{N} \frac{k \|g(x_k)\|_*^2}{k+1} \leq \varepsilon.
\end{equation}
Правая часть предыдущего неравенства может оказаться существенно меньшей, чем для оценки \eqref{eq:2}.
\end{remark}

\section{Адаптивный проксимальный зеркальный метод для вариационных неравенств с относительно сильно монотонными и относительно гладкими операторами} \label{sect_UMP}

Настоящий пункт посвящён методике решения вариационных неравенств с относительно гладкими и относительно сильно монотонными операторами. Предлагаемый подход основан на процедуре рестартов (перезапусков) адаптивного варианта проксимального зеркального метода (алгоритм \ref{Alg:UMP}), недавно предложенного в \cite{UMP}. В (\cite{Inex}, раздел 4) замечено, что этот алгоритм применим к вариационным неравенствам с монотонными и относительно гладкими операторами с сохранением оценки $O(N^{-1})$ скорости роста качества выдаваемого решения по мере роста количества итераций $N$. Мы же покажем, как на базе условия относительной сильной монотонности (определение \ref{DefRelStrongMonot}) можно ввести процедуру рестартов (алгоритм \ref{Alg:RUMP}) алгоритма \ref{Alg:UMP}, которая позволит обосновать линейную скорость сходимости.

\begin{algorithm}[htp]
\caption{Адаптивный проксимальный зеркальный метод для вариационных неравенств.}
\label{Alg:UMP}
\begin{algorithmic}[1]
   \REQUIRE $\varepsilon > 0$, $x_0 \in Q$, $L_0 >0$, $d(x)$, $V(x,z)$.
   \STATE  $k=0$, $z_0 = \arg\min\limits_{u \in Q} d(u)$.
   \STATE  $k:= k+1$.
			\STATE Найти наименьшее $i_k\geq 0:$ 
			\begin{equation}\label{eq_lem_3}
			\begin{multlined}
			    \langle g(z_k) - g(w_k),z_{k+1}-w_k\rangle \leq  L_{k+1}(V(w_k,z_k) + V(z_{k+1},w_k)),
			    \end{multlined}
			\end{equation}
	где $L_{k+1} = 2^{i_k-1}L_k$, причём 
    			\begin{equation}\label{eq_sp_1}
    			     w_k = \arg\min\limits_{x \in Q}\{\langle g(z_k),x-z_k\rangle + L_{k+1}V(x,z_k)\},
    			\end{equation}
    				\begin{equation}\label{alg1_eq2}
    			     z_{k+1} = \arg\min\limits_{x \in Q}\{\langle g(w_k),x-w_k\rangle + L_{k+1}V(x,z_k)\}.
    			\end{equation}
\STATE Переход к п. 2.
\ENSURE $z_k$.
\end{algorithmic}
\end{algorithm}

Рассмотрим вспомогательное утверждение о поведении генерируемой алгоритмом \ref{Alg:UMP} итеративной последовательности, часть доказательства которого (неравенство \eqref{equat3.7}) базируется на (\cite{Inex}, теорема 4.8). При этом, в отличие от \cite{UMP,Inex}, мы предполагаем существование точного решения $x_*$ вариационного неравенства \eqref{eq:1} на множестве $Q$ и оцениваем <<расcтояние>> от выдаваемой методом элементов последовательности $z_N$ до ближайшего к точке старта решения $x_*$ вариационного неравенства \eqref{eq:1}.

\begin{lemma}\label{bregmanlemma}
Пусть  $g$ --- монотонный и относительно гладкий оператор. Тогда для алгоритма \ref{Alg:UMP} верно неравенство
\begin{equation}\label{eqthm1}
V(x_*,z_{N})\leq V(x_*,z_{0}),
\end{equation}
где $x_*$ --- ближайшее к $z_0$ точное решение вариационного неравенства \eqref{eq:1}.
\end{lemma}
\begin{proof}
Нетрудно показать, что в силу \eqref{eq_sp_1} и \eqref{alg1_eq2} для $(k+1)$-ой итерации алгоритма \ref{Alg:UMP}, для всякого $ u\in Q$ имеют место следующие неравенства:
\begin{equation}\label{eq_lem_1}
    \langle g(z_k),w_k-z_k \rangle\leq \langle g(z_k),u-z_k \rangle +L_{k+1}V(u,z_k) - L_{k+1}V(u,w_k)-L_{k+1}V(w_k,z_k),
\end{equation}
\begin{equation}\label{eq_lem_2}
\begin{aligned}
\langle g(w_k),z_{k+1}-w_k \rangle\leq \langle g(w_k),u-w_k \rangle &+L_{k+1}V(u,z_k) - L_{k+1}V(u,z_{k+1}) -\\&
\quad \quad \quad \quad\quad\quad \quad\quad\quad \quad \; -L_{k+1}V(z_{k+1},z_k).
\end{aligned}
\end{equation}

Используя неравенство \eqref{eq_lem_1} при $u=z_{k+1}$, а также \eqref{eq_lem_3} и \eqref{eq_lem_2}, получаем, что 
\begin{equation}\label{equat3.7}
    -\langle g(w_k),u-w_k \rangle\leq L_{k+1}V(u,z_k)-L_{k+1}V(u,z_{k+1}) \quad \forall k\geq 0.
\end{equation}
В силу относительной сильной выпуклости оператора $g$, слабое решение вариационного неравенства также является сильным в следующем смысле:
\begin{equation}
-\langle g(x_*),w_k-x_*\rangle\leq 0,
\end{equation}
что приводит при $u = x_*$ к неравенствам
\begin{equation}
0\leq -\langle g(x_*),x_*-w_k\rangle \leq-\langle g(w_k),x_*-w_k \rangle \leq L_{k+1}V(x_*,z_k)-L_{k+1}V(x_*,z_{k+1}).
\end{equation}

Таким образом, имеют место соотношения
\begin{equation}\label{eqbregmanzzN}
V(x_*,z_{k+1})\leq V(x_*,z_{k}) \quad \forall k \geqslant 0,
\end{equation}
откуда и следует \eqref{eqthm1}.
\end{proof}

Теперь покажем, как соотношения \eqref{eqbregmanzzN}, а также относительная сильная монотонность оператора $g$ (определение \ref{DefRelStrongMonot}) позволяет предложить процедуру рестартов (алгоритм \ref{Alg:RUMP}) алгоритма \ref{Alg:UMP}, гарантирующую достижение заданного качества приближённого решения за линейное время. Отметим, что идеология организации за счёт предположения сильной выпуклости задачи процедуры рестартов (перезапусков) оптимального метода для выпуклых задач оптимизации и вывода лучших оценок скорости сходимости уже на классе сильно выпуклых задач возникла ещё в 1980-х годах (см. например \cite{Nem_Nest_1985}) и нередко используется. Однако новизна предлагаемого подхода заключается в его использовании для относительно сильно монотонных вариационных неравенств в сочетании с адаптивным правилом перезапуска алгоритма \ref{Alg:UMP} (см. пункт 3 листинга алгоритма \ref{Alg:RUMP}).

\begin{algorithm}[htp]
	\caption{Рестарты адаптивного проксимального зеркального метода: метод для относительно гладких и сильно монотонных ВН.}
	\label{Alg:RUMP}
	\begin{algorithmic}[1]
		\REQUIRE $\varepsilon > 0$, $\mu >0$, $\Omega$ : $d(x) \leq \frac{\Omega}{2} \ \forall x\in Q: \|x\| \leq 1$; $x_0,\; R_0 \ : V(x_*,x_0) \leq R_0^2.$
		\STATE $p=0,d_0(x)=R_0^2d\left(\frac{x-x_0}{R_0}\right)$.
		\REPEAT
		\STATE $x_{p+1}$ --- результат работы алгоритма 1 с прокс-функцией $d_{p}(\cdot)$ и критерием остановки $S_N :=\sum_{i=0}^{N-1}L_{i+1}^{-1}\geq \frac{\Omega}{\mu}$.
		\STATE $R_{p+1}^2 = \frac{\Omega R_0^2}{2^{(p+1)}\mu S_{N_p}}$.
		\STATE $d_{p+1}(x) \leftarrow R_{p+1}^2d\left(\frac{x-x_{p+1}}{R_{p+1}}\right)$.
		\STATE $p=p+1$.
		\UNTIL			
		$p > \log_2\left(\frac{2R_0^2}{\varepsilon}\right).$	
		\ENSURE $x_p$.
	\end{algorithmic}
\end{algorithm}

\begin{remark}
Сделаем комментарий о работе алгоритма 2. Для заданной начальной точки  $x_0$, а также $R_0$ и $d_0(x)$ применяется алгоритм 1 с критерием остановки п.3 листинга алгоритма 2. После не более, чем $N_0 = \left \lceil\frac{2L\Omega}{\mu} \right\rceil$ итераций алгоритма 1 мы фиксируем выходную точку $x_1$, обновляем $R_1$ и $d_1(x)$ и производим процедуру рестартов (перезапусков) алгоритма 1 с начальной точки $x_1$. Описанная процедура повторяется до тех пор, пока не будет выполнено количество рестартов алгоритма 1 согласно  п. 7 листинга алгоритма 2.
\end{remark}

Следующая теорема характеризует оценку сложности алгоритма 2.

\begin{theorem}\label{th_restarts}
Пусть оператор $g$ является относительно $\mu$-сильно монотонным и относительно $L$-гладким при $L > 0$ и $\mu > 0$. Тогда для точки выхода $x_p$ алгоритма \ref{Alg:RUMP} имеет место следующая оценка
\begin{equation}\label{equatsummary}
V(x_*,x_p)\leq \varepsilon.
\end{equation}
При этом общее количество итераций алгоритма 1 не превосходит $$N = \left\lceil \frac{2L\Omega}{\mu}\log_2\frac{R_0^2}{\varepsilon} \right\rceil.$$
\end{theorem}
\begin{proof}
В силу относительной сильной монотонности оператора:
\begin{equation}
    \mu V(x_*,z_{p+1})\leq  - \langle g(x_*),z_{p+1}-x_* \rangle - \langle g(z_{p+1}),x_*-z_{p+1}\rangle \leq \\\ - \langle g(z_{p+1}),x_*-z_{p+1} \rangle.
\end{equation}

Пусть $p=0$. Тогда после не более, чем $N_0 = \left\lceil \frac{2L\Omega}{\mu}\right\rceil$ итераций алгоритма 1, справедливы соотношения
\begin{equation}\label{ineqproof}
     -\langle g(x_*), w_{N_0-1}-x_*\rangle \leq -\frac{1}{S_{N_0}}\sum\limits_{k=0}^{N_0-1}\frac{\langle g(w_k),x_*-w_k\rangle}{L_{k+1}}\leq \frac{\Omega R_0^2}{2S_{N_0}} := R_1^2,
\end{equation}
где $w_{N_0-1} = x_1$ согласно обозначениям алгоритма 2. Предположим, что данное неравенство выполнено для некоторого $p>0$:
\begin{equation}\label{res_est}
    \frac{\mu}{S_{N_p}}\sum\limits_{k=0}^{N_p-1}\frac{V(x_*,w_k)}{L_{k+1}}\leq  -\frac{1}{S_{N_p}}\sum\limits_{k=0}^{N_p-1}\frac{\langle g(w_k),x_*-w_k\rangle}{L_{k+1}}\leq \frac{\Omega R_p^2}{2S_{N_p}}.
\end{equation}

Покажем, что неравенство \eqref{res_est} верно и для $p+1$. Действительно, в силу неравенства \eqref{eqthm1} из леммы \ref{bregmanlemma}, имеем:
$$
\sum_{k=0}^{N_{p}-1} \frac{V\left(x_{*}, w_{k}\right)}{L_{k+1}} \geqslant V\left(x_{*}, w_{N_p-1}\right)\cdot S_{N_p}.
$$
Поэтому
$$\frac{\mu}{S_{N_p}} \sum_{k=0}^{N_{p}-1} \frac{V\left(x_{*}, w_{k}\right)}{L_{k+1}} \geqslant \mu V\left(x_{*}, w_{N_{p-1}}\right),$$
откуда в силу пунктов 3 и 4 листинга алгоритма \ref{Alg:RUMP} получаем
$$
V\left(x_{*}, x_{p+1}\right) \leqslant \frac{\Omega R_{p}^{2}}{ 2\mu S_{N_{p}}} = R_{p+1}^2.
$$
Неравенство \eqref{equatsummary} теперь следует из сделанного допущения о количестве рестартов алгоритма 1 при реализации алгоритма 2.
\end{proof}

\begin{remark}
Адаптивность предложенного метода заключается в том, что для запуска алгоритма 2 не требуется информация о константе относительной гладкости $L$, которая фигурирует лишь в оценке скорости сходимости.
\end{remark}

\begin{remark}
Для корректного использования алгоритма 2 в рамках решения задачи о минимизации эмпирического риска \cite{Hendr} необходимо сделать дополнительное допущение о том, что приведенная ранее прокс-функция \eqref{prox_risk} (или, что эквивалентно, функционал центрального сервера $\overline{f}$) ограничена на единичном шаре.
\end{remark}

\section{Оценки скорости сходимости для относительно сильно выпукло-вогнутых седловых задач} \label{section_minmax}

Хорошо известно, что вариационные неравенства с монотонными операторами естественно возникают при рассмотрении выпукло-вогнутых седловых задач, важных для самых разных прикладных проблем. Поэтому в данном пункте статьи мы покажем, как полученные в предыдущих пунктах результаты о методах для вариационных неравенств можно применить к седловым задачам вида
\begin{equation}\label{eqsedlo}
f^* = \min_{u \in Q_1} \max_{v \in Q_2} f(u, v),
\end{equation}
где $f$ --- относительно сильно выпукла по $u$ и относительно сильно вогнута по $v$.

Как известно, необходимость решения вариационных неравенств мотивируется, в частности, как раз задачами вида \eqref{eqsedlo}. В качестве примера можно рассмотреть лагранжеву седловую задачу, порожденную задачей относительно сильно выпуклого программирования.  
\begin{example} Рассмотрим задачу относительно сильно выпуклого программирования (все функционалы $\widehat{f}, g_1, g_2, \ldots, g_m$ относительно сильно выпуклы):
	\begin{equation}\label{problem_with_fun_constraints}
	\left\{\begin{array}{c}
	\min_{x \in Q} \widehat{f}(x), \\
	g_{1}(x), g_{2}(x), \ldots, g_{m}(x) \leq 0.
	\end{array}\right.
	\end{equation}
	
Рассмотрим соответствующую \eqref{problem_with_fun_constraints} лагранжеву седловую задачу следующего вида
	\begin{equation}\label{lagrange_problem}
	\min_{x \in Q} \max_{ \boldsymbol{\lambda}= (\lambda_1, \ldots, \lambda_m)^T \in \mathbb{R}_+^m} L(x, \boldsymbol{\lambda}) :=  \widehat{f}(x) + \sum_{p=1}^{m} \lambda_p g_p(x) - \varepsilon \sum_{p=1}^m \lambda_{p}^2.
	\end{equation}
\end{example}

Для данного типа задач можно ввести такой аналог дивергенции Брэгмана \cite{Fedor_relative_adapuniv}:
$$
V_{\text{new}}\left((y, \boldsymbol{\lambda}), (x, \boldsymbol{\lambda}^{'})\right) = V(y,x) + \frac{1}{2} \left\|\boldsymbol{\lambda} - \boldsymbol{\lambda}^{'}\right\|_2^2, \quad  \forall y, x \in Q, \boldsymbol{\lambda},  \boldsymbol{\lambda}^{'} \in \mathbb{R}_+^m.
$$
Введённая таким образом дивергенция Брэгмана позволяет ослабить требования к ограничениям $\boldsymbol{\lambda}$. 

Приведём некоторые примеры задач оптимизации с относительно сильно выпуклыми функционалами. Начнём с примера относительно гладкого и относительно сильно выпуклого функционала.

\begin{example} (\cite{Lu_Nesterov_2018}) Пусть $\widehat{f}(x):=\frac{1}{4}\|E x\|_{2}^{4}+\frac{1}{4}\|A x-b\|_{4}^{4}+$ $\frac{1}{2}\|C x-d\|_{2}^{2}$, где $A, C$ и $E$ --- положительно определённые квадратные матрицы $n \times n$, $b,d \in \mathbb{R}^n$ --- векторы размерности $n$, $\|\cdot\|_4$ --- норма в пространстве $\ell_4$. В \cite{Lu_Nesterov_2018} показано, что $\widehat{f}$ --- относительно $L$-гладкая и относительно $\mu$-сильно выпукла для
$d(x):=\frac{1}{4}\|x\|_{2}^{4}+\frac{1}{2}\|x\|_{2}^{2}$ на множестве $Q=\mathbb{R}^{n}$, где $L=3\|E\|_{2}^{4}+3\|A\|_{2}^{4}+6\|A\|_{2}^{3}\|b\|_{2}+3\|A\|_{2}^{2}\|b\|_{2}^{2}+\|C\|_2^{2}$ и $\mu=\min \left\{\frac{\sigma_{E}^{4}}{3}, \sigma_{C}^{2}\right\}$, причём $\sigma_{E}$ и $\sigma_{C}$ --- наименьшие собственные значения матриц $E$ и $C$.
\end{example}

Теперь приведём пример относительно липшицева и относительно сильно выпуклого функционала.

\begin{example} (\cite{Zhou_NIPS_2020}) 
	Пусть
	$\widehat{f}(x) := \frac{1}{p} \|x\|_2^p$
	для $p \geq 2$ и $ Q = [-\alpha, \alpha]^n, \; \alpha > 0$. Заметим, что $\nabla \widehat{f}(x) = \|x\|_2^{p - 2} x$ и $\nabla^2 \widehat{f}(x) = \|x\|_2^{p - 2} I + (p-2)\|x\|_2^{p - 4} x x^{T}$  ($I$ --- единичная матрица). Тогда $\widehat{f}$ является относительно $M$-липшицевой при $M = 1$ относительно прокс-фyнкции $d(x) := \frac{1}{2p}\|x\|_2^{2p}$. При этом $\widehat{f}$ не сильно выпукла в обычном смысле, т.к. $\nabla^2 \widehat{f}(x) - \mu I$ не является положительно полуопределённой в окрестности 0 ни при каком $\mu >0$. Тем не менее, при 
	\begin{equation}\label{eq_mu}
	 \mu = \frac{p-1}{(2p - 1)(\sqrt{n}\alpha)^p}  
	\end{equation}
матрица $\nabla^2 \widehat{f}(x) - \mu \nabla^2 d(x)$ уже будет положительно полуопределённой, что влечёт относительную сильную выпуклость $\widehat{f}$.
	\label{ex_experiments}
\end{example}

Перейдём теперь к методике для нахождения приближённого решения задачи \eqref{eqsedlo}. Для всякого $\varepsilon > 0$ под $\varepsilon$-точным решением задачи \eqref{eqsedlo} будем понимать пару $(\widehat{u}, \widehat{v})$ такую, что $$\max_{v \in Q_2} f(\widehat{u}, v) - \min_{u \in Q_1} f(u, \widehat{v}) \leq \varepsilon.$$ Обозначим $x = (u, v), y = (z, t)$, а также введем оператор 
\begin{equation}\label{operator-sedlo}
g(x) := \Bigg( 
\begin{aligned}
f^{'}_{u}(u,v)\\
-f^{'}_{v}(u,v)
\end{aligned}
\Bigg).
\end{equation}
Тогда ввиду выпукло-вогнутости $f$ имеем: 
\begin{equation}
\begin{aligned}
\langle g(x), x - y \rangle &=
 \Bigg( 
\begin{aligned}
f^{'}_{u}(u,v)\\
-f^{'}_{v}(u,v)
\end{aligned}
\Bigg)
 (u - z, v - t)  = \langle f^{'}_{u}(u,v), u - z \rangle - \langle f^{'}_{v}(u,v), v - t \rangle \geq \\&
 \geq f(u, v) - f(z, v) 
- f(u, v)+ f(u, t)=  f(u,t) - f(z, v).
\end{aligned}
\end{equation}

Будем предполагать относительную ограниченность оператора \eqref{operator-sedlo}. Тогда метод \eqref{eq:4} для задач \eqref{eqsedlo} приводит к оценкам вида:
\begin{equation} \label{eq:21}
\sum_{k=1}^{N} \frac{2k}{N(N+1)} \langle g(x_k), x_k -x\rangle \leq \frac{2 M^2}{\mu (N+1)}.
\end{equation}
Если $x_k = (u_k, v_k), \;\; x = (u, v)$, то  
\begin{equation}
\langle g(x_k), x_k -x\rangle \geq f(u_k,v) - f(u, v_k) \quad \forall (u, v).
\end{equation}
Далее, \eqref{eq:21} означает, что 
\begin{equation}
\sum_{k=1}^{N} \frac{2k}{N(N+1)} (f(u_k,v) - f(u, v_k)) \leq \frac{2M^2}{\mu (N+1)}.
\end{equation}
Положим
\begin{equation}
(\widehat{u}, \widehat{v}) := \frac{1}{N(N+1)} \sum_{k=1}^{N} 2k (u_k,v_k).
\end{equation}
Тогда получаем, что
\begin{equation}
\sum_{k=1}^{N} \frac{2k}{N(N+1)} (f(u_k, v) - f(u, v_k)) \geq f(\widehat{u}, v) - f(u, \widehat{v}), 
\end{equation}
откуда ввиду \eqref{eq:2} получаем для задачи \eqref{eqsedlo} следующую оценку
\begin{equation}
\max_{v} f(\widehat{u}, v) - \min_{u} f(u, \widehat{v}) \leq \frac{2M^2}{\mu (N+1)}.
\end{equation}

\begin{remark}
Оценка теоремы \ref{th_restarts} применима к сильно выпукло-вогнутой седловой задаче \eqref{eqsedlo} в случае, если порождённый ей оператор \eqref{operator-sedlo} удовлетворяет условию относительной гладкости \eqref{rel_smooth}.
\end{remark}

\section{Заключение}
В данной статье было введено понятие относительной сильной монотонности операторов вариационного неравенства и рассмотрены два типа предположений о непрерывности (гладкости) операторов: относительная ограниченность и относительная гладкость. Для решения вариационных неравенств указанных типов были исследованы следующие численные методы градиентного типа: вариант субградиентного метода (зеркального спуска), а также рестартованный адаптивный проксимальный метод. Стоит отметить, что субградиентный метод \eqref{eq:4} оптимален на классе вариационных неравенств с относительно сильно монотонными и относительно ограниченными операторами, как и на классе относительно сильно выпуклых и относительной липшицевых задач минимизации \cite{Lu_2018,Zhou_NIPS_2020}. В то же время для рестартованного варианта адаптивного проксимального метода (алгоритм \ref{Alg:RUMP}) доказывается линейная скорость сходимости на классе вариационных неравенств с относительно сильно монотонными и относительно гладкими операторами. Также были проведены численные эксперименты, демонстрирующие преимущества в использовании предложенных методов. В частности, скорость сходимости метода \eqref{eq:4} численно сравнивалась с недавно предложенным адаптивным алгоритмом (алгоритм 2 из  \cite{Fedor_relative_adapuniv}). Эксперименты проводились для оператора $g = \nabla F$, где относительно липшицева и относительно сильно выпуклая функция $F$ выбрана согласно примеру \ref{ex_experiments} (более подробное описание можно найти в \cite{Zhou_NIPS_2020}), результаты экспериментов приведены в рисунке \ref{fig_numerical}. Анализируя полученные результаты, стоит отметить, что учёт относительной сильной монотонности с параметром \eqref{eq_mu} может привести  к тому, что  предложенный в настоящей работе алгоритм \eqref{eq:4} на практике работает лучше, чем адаптивный метод \cite{Fedor_relative_adapuniv}, не учитывающий свойство относительной сильной монотонности оператора $g$.

\begin{figure}[htp]
\begin{center}
	\minipage{0.50\textwidth}
	\includegraphics[width=\linewidth]{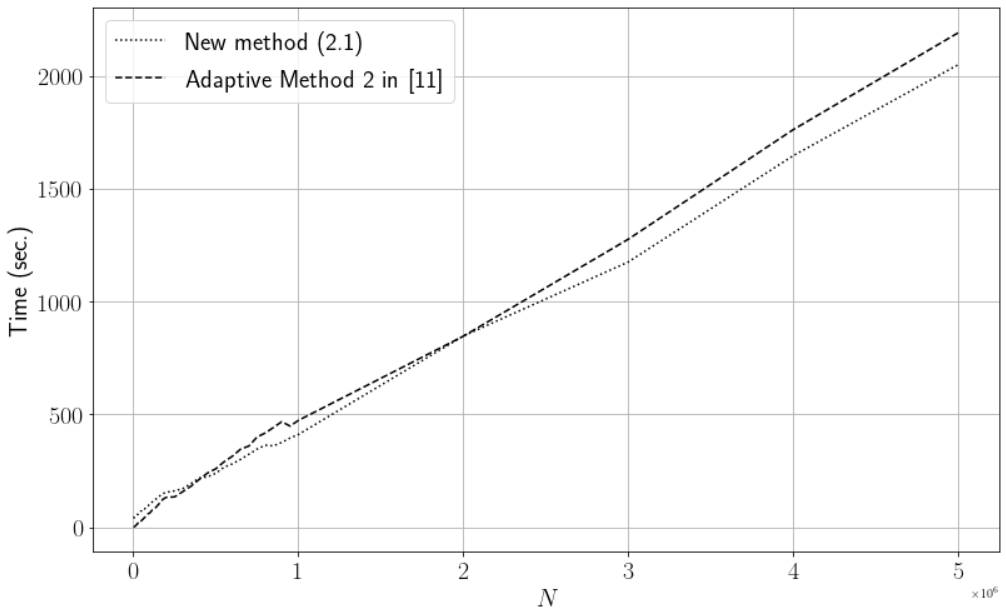}
	\endminipage
	\minipage{0.50\textwidth}
	\includegraphics[width=\linewidth]{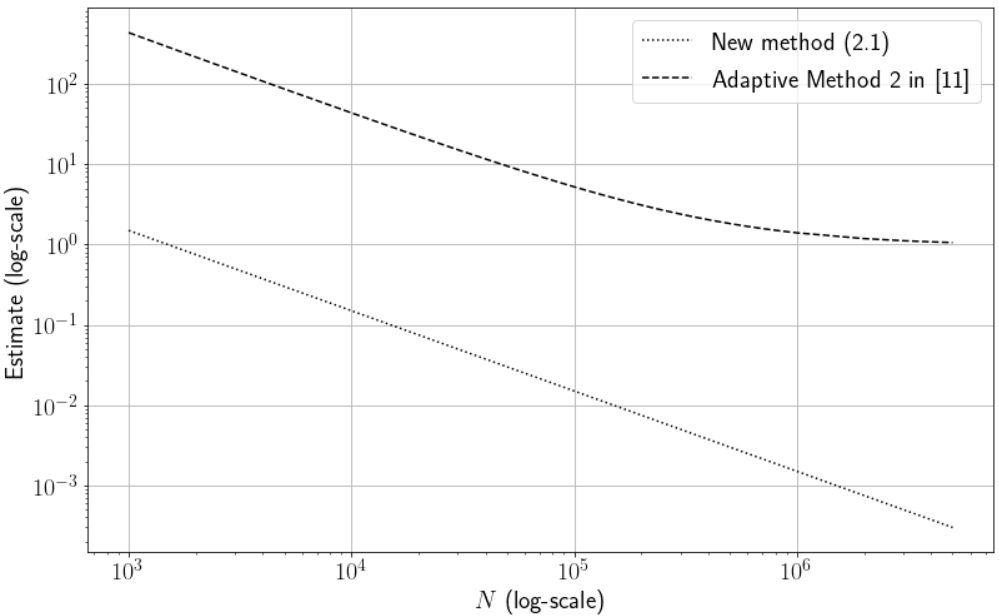}
	\endminipage
\end{center}
\caption{Результаты сравнения для $n=1000, \alpha =0.5, p = 2.$}
\label{fig_numerical}
\end{figure}

\end{document}